\documentclass[12pt,reqno]{article}

\usepackage[usenames]{color}
\usepackage{amssymb}
\usepackage{graphicx}
\usepackage{amscd}

\usepackage[colorlinks=true,
linkcolor=webgreen,
filecolor=webbrown,
citecolor=webgreen]{hyperref}

\definecolor{webgreen}{rgb}{0,.5,0}
\definecolor{webbrown}{rgb}{.6,0,0}

\usepackage{color}
\usepackage{fullpage}
\usepackage{float}

\usepackage{graphics,amsmath,amssymb}
\usepackage{amsthm}
\usepackage{amsfonts}
\usepackage{latexsym}
\usepackage{epsf}

\newcommand{\seqnum}[1]{\href{http://oeis.org/#1}{\underline{#1}}}

\begin{document}

\theoremstyle{plain}
\newtheorem{theorem}{Theorem}
\newtheorem{corollary}[theorem]{Corollary}
\newtheorem{lemma}[theorem]{Lemma}
\newtheorem{proposition}[theorem]{Proposition}

\theoremstyle{definition}
\newtheorem{definition}[theorem]{Definition}
\newtheorem{example}[theorem]{Example}
\newtheorem{conjecture}[theorem]{Conjecture}
\newtheorem{problem}[theorem]{Problem}

\theoremstyle{remark}
\newtheorem{remark}[theorem]{Remark}

\begin{center}
\vskip 1cm{\LARGE\bf
Digit Preserving Multiplication
\vskip .11in
in Continued Fraction Representations
}
\vskip 1cm
\large
Benjamin V. Holt \\
Humboldt State University\\
1 Harpst Street \\
Arcata, CA 95521 \\
USA\\
\href{mailto:bvh6@humboldt.edu}{\tt bvh6@humboldt.edu}\\
\end{center}

\begin{abstract}
A permutiple is a number which is an integer multiple of some permutation of its digits. 
A well-known example is 9801 since it is an integer multiple of its reversal, 1089.
In this paper, we consider the permutiple problem in an entirely different setting: continued fractions.
We pose the question of when the simple continued fraction representation of a rational number is
an integer multiple of a permutation of its partial quotients (or digits, as we shall call them).
We develop some general results and apply them to finding new examples.
In doing so, we attempt to classify all 2, 3, and 4-digit continued fraction permutiples in terms of basic permutiple types 
which we discover along the way. We also generate new examples from old by finding conditions which guarantee
that digit-string concatenation yields other permutiples.
\end{abstract}

\section{Introduction} 
Representations of numbers are a topic of fascination for both amateur and professional mathematicians alike.
A well-known case in point is the following puzzle: find digits $A$, $B$, $C$, and $D$ such that
$4\cdot ABCD=DCBA$. This problem, along with its solution, $4\cdot 2178=8712$,
has appeared in its fair share of puzzle books and columns.
In his oft quoted essay, \textit{A Mathematician's Apology}, G. H. Hardy uses  
such numbers (known as both \textit{reverse multiples} \cite{kendrick_1,sloane,young_1} 
and \textit{palintiples} \cite{holt,holt_2,holt_3}) to illustrate 
an uninteresting theorem; 2178 and 1089 are the only four-digit numbers whose reversal yields an integer
multiple of itself. He goes on to tell his readers that such a result is
``suitable for puzzle columns,'' but ``not capable of any significant generalization.'' 
As many have been quick to point out \cite{holt,kendrick_1,pudwell,young_1}, this article being no exception, 
Hardy was perhaps mistaken in his latter assessment; 
the problem, as demonstrated by the papers cited thus far, generalizes quite naturally and has led to other areas of research
including the notion of \textit{Young graphs} \cite{kendrick_1, sloane} 
which describe the digit-carry structure of these numbers. 

Other digit permutation problems have received professional attention too. 
Cyclic digit permutations such as $5 \cdot 142857=714285$ are also well-studied \cite{guttman,kalman},
and, at least in comparison to reverse multiples, are well-understood.
We note that we are aware of only one work \cite{holt_3} which addresses general digit permutations.
The reason for such lack of generality is 
that even for a particular permutation, solving the problem can be far from trivial. We take reverse multiples as an example;
only a small number of reverse multiple types, as determined by Young graph isomorphism,
are well understood, leaving a multitude of others as yet unexplored 
\cite{holt_2, kendrick_1, kendrick_2}.

In this paper, we shift the setting from radix representations of integers
to continued fractions; we ask when the continued fraction representation 
of a rational number is an integer multiple of a continued fraction with the same partial quotients 
(which for the duration of this paper we shall refer to as digits).
We first consider some general properties of these new permutiples, and with our results, generate 
new examples. Along the way, we discover several fundamental permutiple types and attempt to classify 
all permutiples of 4 digits or less in terms of these types.
We then consider methods for obtaining new permutiples from old. In particular, we focus on the problem of when
digit-string concatenation yields new permutiples.
In our conclusion, we briefly consider the infinite case and survey analogous ideas which carry over from finite permutiples.
We shall only consider simple continued fractions leaving more general representations to the ambitious reader.

\section{Problem formulation and basic results}

\begin{definition}
 Using the notation $[a_0; a_{1}, \ldots, a_n]$ to represent the finite simple continued fraction
\[
 a_0+\frac{1}{a_{1}+\frac{1}{\ddots+\frac{1}{a_{n}}}},
\]
we let $k>1$ be a natural number and $\sigma$ be a permutation on $n+1$ symbols.
Then the number $[a_0; a_{1}, \ldots, a_n]$ is a \textit{continued fraction $(\sigma,k)$-permutiple} provided
\[
 [a_0; a_{1}, \ldots, a_n]=k[a_{\sigma(0)}; a_{\sigma(1)}, \ldots, a_{\sigma(n)}].
\]
\end{definition}

Since there will be no possibility for confusion between permutiples in continued fraction and radix representation settings, 
we shall, for the sake of cleaner exposition, refer to {\it continued fraction $(\sigma,k)$-permutiples} as simply
{\it $(\sigma,k)$-permutiples}.
In this same spirit, using terminology similar to Young \cite{young_1}, Sloane \cite{sloane}, and Kendrick \cite{kendrick_1}, 
we shall refer to continued fraction $(\sigma,k)$-permutiples for which $\sigma$ is the reversal permutation 
as {\it $k$-reverse multiples}.

\begin{example}
The continued fraction $[7;1,3]$ is a 2-reverse multiple since $[7;1,3]=2\cdot [3;1,7]$.
\end{example}

We shall also keep in mind a sequence related to the digits of our continued fraction expansion; 
defining $G:[0,\infty)\rightarrow [0,1)$ by $G(0)=0$ and 
$G(x)=\frac{1}{x}-\lfloor \frac{1}{x} \rfloor$ when $x>0$ (left shift operator),
we find the continued fraction expansion, $[a_0; a_1, a_2, \ldots]$, of a real number $r$ by first
generating the orbit of $\gamma_0=r-\lfloor r \rfloor$ under $G$.
We then obtain the digits of our continued fraction from the orbit $\{\gamma_0$, $\gamma_1$, $\ldots\}$
according to the rule $a_0=\lfloor r \rfloor$, and $a_{j+1}=\lfloor \frac{1}{\gamma_j}\rfloor$.
We shall be particularly interested in how $(\sigma,k)$-permutiples relate to their corresponding orbits under $G$.

For the remainder of this paper, we shall take $\frac{p_j}{q_j}$ and $\frac{p'_j}{q'_j}$ to be the $j$th 
convergents of $r=[a_0; a_{1}, \ldots, a_n]$ and $r'=[a_{\sigma(0)}; a_{\sigma(1)}, \ldots, a_{\sigma(n)}]$, respectively.
We also take  $\{\gamma_0$, $\gamma_1$, $\gamma_2$, $\ldots\}$ and $\{\gamma'_0$, $\gamma'_1$, $\gamma'_2$, $\ldots\}$
to be the orbits of $\gamma_0=r-\lfloor r \rfloor$ and $\gamma'_0=r'-\lfloor r' \rfloor$ under $G$, respectively.
We shall refer to these sequences as the {\it tails} of $r$ and $r'$, respectively.
Finally, so that finite continued fraction representations are unique, we shall assume that every finite continued
fraction is in canonical form.

Denoting the $n$th continuant as $K_{n}(x_0,x_1,\ldots,x_{n-1})$,
we may restate the problem as finding natural numbers $a_0$, $a_1$, $\ldots$ $a_n$, and a permutation $\sigma$ such that
\begin{equation}
\frac{p_n}{q_n}=\frac{K_{n+1}(a_0,a_1,\ldots,a_n)}{K_{n}(a_1,\ldots,a_n)}
=k\frac{K_{n+1}(a_{\sigma(0)},a_{\sigma(1)},\ldots,a_{\sigma(n)})}{K_{n}(a_{\sigma(1)},\ldots,a_{\sigma(n)})}
=k\frac{p'_n}{q'_n}
\label{continuant}
\end{equation}
for some natural number $k>1$. In the coming pages, the reader shall see that
the next definition will be of central importance to our effort.

\begin{definition} 
 We say that a $(\sigma,k)$-permutiple $[a_0; a_{1}, \ldots, a_n]$ is {\it continuant-preserving} 
 if $$K_{n+1}(a_0,a_1,\ldots,a_n)=K_{n+1}(a_{\sigma(0)},a_{\sigma(1)},\ldots,a_{\sigma(n)}).$$
 \label{contin_pres_def}
\end{definition}

\begin{remark}
 Since the reversal permutation preserves continuant polynomials in general, all $k$-reverse multiples are
continuant-preserving. 
\end{remark}

\begin{theorem}
 If $r=[a_0; a_{1}, \ldots, a_n]$ is a $(\sigma,k)$-permutiple, then the following are equivalent:  
 \begin{enumerate}
  \item $r$ is continuant-preserving,
  \item $K_{n}(a_{\sigma(1)},\ldots,a_{\sigma(n)})=k K_{n}(a_1,\ldots,a_n)$,  
  \item  $\gamma_0 \gamma_1 \cdots \gamma_{n-1}=k \gamma'_0 \gamma'_1 \cdots \gamma'_{n-1}$.
 \end{enumerate}
  \label{contin_pres_equiv}
\end{theorem}

\begin{proof}
  Equation \ref{continuant} makes obvious the equivalence of items {\it 1} and {\it 2} above. 
  We shall prove $({\it 2})\Longleftrightarrow({\it 3})$.
 By induction we have  $\gamma_0 \gamma_1 \cdots \gamma_j=\frac{1}{q_{j}/\gamma_{j}+q_{j-1}}$ for $j<n$.
 Hence, $\gamma_0 \gamma_1 \cdots \gamma_{n-1}=\frac{1}{a_n q_{n-1}+q_{n-2}}=\frac{1}{q_n}$. 
 Similarly, $\gamma'_0 \gamma'_1 \cdots \gamma'_{n-1}=\frac{1}{q'_n}$. 
\end{proof}

\begin{theorem}
 If $r=[a_0; a_{1}, \ldots, a_n]$ is a $(\sigma,k)$-permutiple such that $\frac{p_n}{p'_n}<2$, then $r$ is continuant-preserving.
  \label{basic}
 \end{theorem}
 \begin{proof}
 Since $\gcd(p_n,q_n)=1$ and $\gcd(p'_n,q'_n)=1$, we have that $p'_n$ divides $p_n$. 
 Hence, since $\frac{p_n}{p'_n}<2$ by assumption, $p'_n=p_n$ so that 
 $K_{n+1}(a_0,a_1,\ldots,a_n)=K_{n+1}(a_{\sigma(0)},a_{\sigma(1)},\ldots,a_{\sigma(n)})$.
\end{proof}
 
 \section{Permutiples of two and three digits}\label{3digs}
 
 We now find all permutiples of 3 digits or less. 
 A trivial calculation shows that the form of any 2-digit permutiple is $[a_0;a_1]=[ks,s]$, where $s$ is an integer
 parameter greater than 1. 
 
 Suppose that $n+1=3$. Since $a_2>1$ by assumption, 
 and since $a_0 > a_{\sigma(0)} \geq 1$ for any $(\sigma,k)$-permutiple, we see that
 $$
 \frac{p_2}{p'_2}
 =\frac{a_0(a_1 a_2+1)+a_2}{a_{\sigma(0)}(a_{\sigma(1)} a_{\sigma(2)}+1)+a_{\sigma(2)}}
 =\frac{1+\frac{1}{a_1 a_2}+\frac{1}{a_0 a_1}}{1+\frac{a_{\sigma(0)}}{a_{0}a_{1}a_{2}}
 +\frac{a_{\sigma(2)}}{a_{0}a_{1}a_{2}}}
 < 1+\frac{1}{2 a_1}+\frac{1}{2 a_1} \leq 2.
 $$
 Hence, $\frac{p_2}{p'_2}<2$ so that by Theorem and \ref{basic}, any 3-digit permutiple is continuant-preserving.
 Thus, 
 \begin{equation}
  a_0+a_2=a_{\sigma(0)}+a_{\sigma(2)}.
  \label{3-digit_sum}
 \end{equation} 
 We now use Equation \ref{3-digit_sum} to narrow down which permutations are possible. Now, $\sigma(0)$ cannot be 0 since
 $a_0 > a_{\sigma(0)}$ for any permutiple. Suppose then that $\sigma(0)=1$. By Equation \ref{3-digit_sum}, we
 have $a_0+a_2=a_{1}+a_{\sigma(2)}$. If $\sigma(2)=2$, then $a_0=a_1$, which is impossible since this would imply that
  $a_0 > a_{\sigma(0)}=a_1=a_0$. Hence, the only possibility left is $\sigma(2)=0$ so that $a_1=a_2$. 
  Under this assumption we have that $r=[a_0;a_1,a_2]=k[a_1;a_2,a_0]=k[a_2;a_1,a_0]$, and thus, 
  in this case, $r$ must be a $k$-reverse multiple. 
  If we assume that $\sigma(0)=2$, then a similar argument demonstrates that our permutiple must again be a $k$-reverse multiple.
  Hence, every 3-digit permutiple is a $k$-reverse multiple.
  
 Now, using Theorem \ref{contin_pres_equiv}, the continued fraction $[a_0;a_1,a_2]$ is a $k$-reverse multiple
 if and only if  $K_{2}(a_{1},a_{0})=k K_{2}(a_1,a_2)$, or  
\begin{equation}
 a_0 a_1+1=k(a_1 a_2 +1). 
\label{3-digit} 
\end{equation}
 Therefore, $a_0 a_1 \equiv -1$ (mod $k$). If we suppose that $k$ divides $a_0 a_1+1$ with quotient $q$,
 then $a_1 a_2 \equiv -1$ (mod $q$). Now, $a_0$ must be greater than $k a_2$. For if we assume that $a_0 \leq k a_2$,
 then by Equation \ref{3-digit}, $k \leq 1$, which is a contradiction of our initial assumption. 
 It then follows by Equation \ref{3-digit} that $a_1 \leq k-1$.
 Finally, since $a_2 a_1+1=q$, we have $a_2<q$.
 
 We now have the following method for finding all 3-digit $k$-reverse multiples:
 choose an $a_0>k$ and relatively prime to $k$. Next, find the multiplicative inverse, $\alpha$, of $a_0$ modulo $k$. 
 Then $a_1=k-\alpha$. Finally, find the multiplicative inverse $0<\beta<q$ of $a_1$ modulo $q$ where $q=\frac{a_0a_1+1}{k}$.
 (A solution exists since $a_1$ and $q$ are relatively prime.) Then $a_2=q-\beta$.
 
 The reader will also notice that this method gives all non-canonical representations as well. 
 For example, for $k=2$, and $a_0=5$, we obtain the non-canonical $2$-reverse multiple $[5;3,1]$.

\section{Perfect permutiples}

For any value of $k>1$, we may construct an infinitude of permutiples with more than 3 digits. We observe that 
\[
 kb_0+\frac{1}{b_{1}+\frac{1}{kb_{2}+\frac{1}{b_3+\frac{1}{\ddots}}}}
=k \left(b_0+\frac{1}{kb_{1}+\frac{1}{b_2+\frac{1}{kb_3+\frac{1}{\ddots}}}}\right)
\]
Thus, setting $a_j=kb_j$ and $a_{\sigma(j)}=b_j$ for even $j$, and $a_j=b_j$ and $a_{\sigma(j)}=kb_j$ for odd $j$, 
we obtain the system $a_j=ka_{\sigma(j)}$ for even $j$, and $k a_{j}=a_{\sigma(j)}$ for odd $j$.
This is a remarkably easy system to solve. As an example, we choose a value of $k$, say $7$, a permutation $\sigma$, 
say $(0,1)(2,3)$, and we find digits $a_0$, $a_1$, $a_2$, and $a_3$ such that $a_0=7a_{1}$, $a_2=7a_{3}$, 
$7 a_{1}=a_{0}$, and $7 a_{3}=a_{2}$. The system simplifies to 
$a_0=7s$, $a_{1}=s$, $a_2=7t$, and $a_{3}=t$, where $s \geq 1$ and $t>1$ are integer parameters.
Then
\[
 a_0+\frac{1}{a_1+\frac{1}{a_{2}+\frac{1}{a_3}}}
= 7s+\frac{1}{s+\frac{1}{7t+\frac{1}{t}}}
=7 \left(s+\frac{1}{7s+\frac{1}{t+\frac{1}{7t}}} \right)
=7 \left(a_1+\frac{1}{a_0+\frac{1}{a_{3}+\frac{1}{a_2}}} \right).
\]
Thus, for example, the reader may verify that $[7,1,14,2]$ is a $((0,1)(2,3),7)$-permutiple.
The above construction motivates the following definition.

\begin{definition}
We say a $(\sigma,k)$-permutiple $[a_0; a_{1}, \ldots, a_n]$ is {\it perfect} provided
\begin{equation}
k=\frac{a_0}{a_{\sigma(0)}}=\frac{a_{\sigma(1)}}{a_1}=\frac{a_2}{a_{\sigma(2)}}=\frac{a_{\sigma(3)}}{a_3}=\cdots.
\label{perfect_eq}
\end{equation} 
\label{perfect_def}
\end{definition}

\begin{remark}
 Every 2-digit permutiple is a perfect reverse multiple.
\end{remark}

As the reader is sure to have noticed when constructing the previous example, 
the choice of the multiplier, $k$, was entirely arbitrary since any suitable value would have illustrated our purpose.
However, the choice of permutation was not arbitrary;
Equation \ref{perfect_eq} puts tight restrictions on the kind of permutation $\sigma$ can be.
The first, and most obvious, is that $\sigma$ must be a {\it derangement}, otherwise, it would be that $k=1$.

We now represent Equation \ref{perfect_eq} in matrix form,
$\mathbf{d}=
\left[
 \begin{matrix}
 k & 0 & 0 & 0 & \cdots\\
 0 & 1/k & 0 & 0 & \cdots\\
 0 & 0 & k & 0 & \cdots\\
 0 & 0 & 0 & 1/k & \cdots\\
 \vdots & \vdots & \vdots & \vdots & \ddots\\
\end{matrix}
\right]
P_{\sigma}\textbf{d},$ where $P_{\sigma}$ is the permutation matrix of $\sigma$, 
and  
$\mathbf{d}
=\left[
\begin{matrix}
 a_0\\
 a_1\\
 \vdots\\
 a_n
\end{matrix}
\right]$. (Note that all of these matrices are indexed from 0 to $n$). 
Now, we see that in the case of even $j$ that
$a_{j}$ must be divisible by $k$ with quotient $s_j$, which we will interpret as a free integer parameter in the solution
as we did in the example at the beginning of this section.
Then we have 
$\textbf{d}=
\left[
 \begin{matrix}
 ks_0 \\
 s_{1}  \\
 ks_{2} \\
 s_{3}\\
 \vdots\\
\end{matrix}
\right]
=
\left[
 \begin{matrix}
 k & 0 & 0 & 0 & \cdots\\
 0 & 1 & 0 & 0 & \cdots\\
 0 & 0 & k & 0 & \cdots\\
 0 & 0 & 0 & 1 & \cdots\\
 \vdots & \vdots & \vdots & \vdots & \ddots\\
\end{matrix}
\right]
\left[
 \begin{matrix}
 s_0 \\
 s_{1}  \\
 s_{2} \\
 s_{3}\\
 \vdots\\
\end{matrix}
\right].
$ 
After substitution and simplification, the above system becomes
$
\left[
 \begin{matrix}
 s_0 \\
 s_{1}  \\
 s_{2} \\
 s_{3}\\
 \vdots\\
\end{matrix}
\right]
=
\left[
 \begin{matrix}
 1 & 0 & 0 & 0 & \cdots\\
 0 & 1/k & 0 & 0 & \cdots\\
 0 & 0 & 1 & 0 & \cdots\\
 0 & 0 & 0 & 1/k & \cdots\\
 \vdots & \vdots & \vdots & \vdots & \ddots\\
\end{matrix}
\right]
P_{\sigma}
\left[
 \begin{matrix}
 k & 0 & 0 & 0 & \cdots\\
 0 & 1 & 0 & 0 & \cdots\\
 0 & 0 & k & 0 & \cdots\\
 0 & 0 & 0 & 1 & \cdots\\
 \vdots & \vdots & \vdots & \vdots & \ddots\\
\end{matrix}
\right]
\left[
 \begin{matrix}
 s_0 \\
 s_{1}  \\
 s_{2} \\
 s_{3}\\
 \vdots\\
\end{matrix}
\right]$.
Using the definition of matrix multiplication, we obtain the component form of the above,
\begin{equation}
s_j=k^{-\frac{1}{2}[(-1)^j+(-1)^{\sigma(j)}]}s_{\sigma(j)}.
\label{parameters}
\end{equation}

Equation \ref{parameters} and an induction argument give us
$s_j=k^{-\frac{1}{2}\sum_{\ell=0}^{m-1}[(-1)^{\sigma^{\ell}(j)}+(-1)^{\sigma^{\ell+1}(j)}]}s_{\sigma^{m}(j)}$
for any $m \geq 1$ and for all $0 \leq j \leq k$. Thus, when $m=|\sigma|$,
\begin{equation}
\sum_{\ell=0}^{|\sigma|-1}(-1)^{\sigma^{\ell}(j)}=0.
\label{orbits} 
\end{equation}
An immediate consequence of Equation \ref{orbits} is that $|\sigma|$ must be even. 
Moreover, by letting  $\langle \sigma \rangle$ act on $S=\{0,1, 2,\cdots, n\}$ in the usual
way, the Orbit-Stabilizer theorem tells us that the size of every orbit equals $|\sigma|$ since
$\sigma$ is a derangement.
Thus, every orbit contains an even number of elements.
Moreover, since the orbits partition $S$, we have that $n+1$ must also be even and divisible by $|\sigma|$ where 
the quotient is equal to the number of orbit classes. 
Our results about allowable permutations can then be summarized by the following.   
\begin{theorem}
 If $[a_0; a_{1}, \ldots, a_n]$ is a perfect $(\sigma,k)$-permutiple, then the following must hold:
 \begin{enumerate}
  \item The permutation $\sigma$ is a derangement of even order.
  \item The order of $\sigma$ divides $n+1$ with quotient equal to the number of orbit classes.
  \item Any orbit of $\sigma$ contains $|\sigma|$ elements, half of which are odd, and half of which are even.
 \end{enumerate}
 \label{derangement}
\end{theorem}

The reversal permutation on an even number of symbols satisfies all three conclusions of Theorem \ref{derangement},
but not all reverse multiples with an even number of digits are perfect. 
Therefore, the converse of Theorem \ref{derangement} does not hold in general;
the 2-reverse multiple, $[7; 2, 1, 3]$, provides a counterexample.

Our next theorem characterizes perfect permutiples in terms of their tails.

\begin{theorem}
 Let $[a_0; a_{1}, \ldots, a_n]$ be a $(\sigma,k)$-permutiple. Then the following are equivalent:
 \begin{enumerate}
  \item $[a_0; a_{1}, \ldots, a_n]$ is perfect,
  \item $k=\frac{\gamma_{0}}{\gamma'_{0}}=\frac{\gamma'_{1}}{\gamma_{1}}=\frac{\gamma_{2}}{\gamma'_{2}}=\frac{\gamma'_{3}}{\gamma_{3}}=\cdots$,
  \item $\frac{1}{k}G(x)=G(kx)$ when  $x=\gamma'_0$, $\gamma_1$, $\gamma'_2$, $\gamma_3, \ldots$, and $G(\frac{1}{k}x)=kG(x)$ when  $x=\gamma_0$, $\gamma'_1$, $\gamma_2$, $\gamma'_3, \ldots$.
 \end{enumerate}
 \label{perf_equiv}
\end{theorem}
\begin{proof} 
Let $r=[a_0; a_{1}, \ldots, a_n]$ and $r'=[a_{\sigma(0)}; a_{\sigma(1)}, \ldots, a_{\sigma(n)}]$.
We establish $({\it 1})\Longrightarrow({\it 2})$ by induction. The first tails are $\gamma_0=r-\lfloor r \rfloor$ and 
$\gamma'_0=r'-\lfloor r' \rfloor$, and since $a_0=ka_{\sigma(0)}$ and $r=kr'$ by assumption, we have $\gamma_0=k\gamma'_0$.
Hence, 
$
k\gamma_1
=\frac{k}{\gamma_0}-k \lfloor \frac{1}{\gamma_0} \rfloor
=\frac{1}{\gamma'_0}-k a_1
=\frac{1}{\gamma'_0}-a_{\sigma(1)}
=\frac{1}{\gamma'_0}-\lfloor \frac{1}{\gamma'_0} \rfloor
=\gamma'_1.
$
Now suppose that $j$ is even and that both $\gamma_j=k\gamma'_j$ and $k \gamma_{j+1}=\gamma'_{j+1}$.
Then 
$
\gamma_{j+2}
=\frac{1}{\gamma_{j+1}}-\lfloor \frac{1}{\gamma_{j+1}} \rfloor
=\frac{1}{\gamma_{j+1}}-a_{j+2}
=\frac{k}{\gamma'_{j+1}}-ka_{\sigma(j+2)}
=\frac{k}{\gamma'_{j+1}}-k\lfloor \frac{1}{\gamma'_{j+1}} \rfloor
=k \gamma'_{j+2}
$
so that
$
k \gamma_{j+3}
=  \frac{k}{\gamma_{j+2}}- k \lfloor \frac{1}{\gamma_{j+2}} \rfloor
= \frac{k}{\gamma_{j+2}}-ka_{j+2}
= \frac{k}{\gamma_{j+2}}-a_{\sigma(j+2)}
= \frac{1}{\gamma'_{j+2}}-\lfloor \frac{1}{\gamma'_{j+2}} \rfloor
= \gamma'_{j+3}.
$

For $({\it 2})\Longrightarrow({\it 3})$, we suppose 
$k=\frac{\gamma_{0}}{\gamma'_{0}}=\frac{\gamma'_{1}}{\gamma_{1}}=\frac{\gamma_{2}}{\gamma'_{2}}=\frac{\gamma'_{3}}{\gamma_{3}}=\cdots$.
Then for even $j$ we have $k \gamma_{j+1}=\gamma'_{j+1}$.
Consequently, 
$
kG(\gamma_j)
=\frac{k}{\gamma_{j}}-k \lfloor \frac{1}{\gamma_{j}} \rfloor
=\frac{1}{\gamma'_{j}}-\lfloor \frac{1}{\gamma'_{j}} \rfloor
=\frac{k}{\gamma_{j}}-\lfloor \frac{k}{\gamma_{j}} \rfloor
=G(\gamma_j/k)
$
which implies 
$kG(k\gamma'_j)
=kG(\gamma_j)
=G(\gamma_j/k)
=G(\gamma'_j)
$ for even $j$.
Similarly, $kG(k\gamma_j)
=kG(\gamma'_j)
=G(\gamma'_j/k)
=G(\gamma_j)
$ when $j$ is odd.

To prove $({\it 3})\Longrightarrow({\it 1})$, we have for even $j$ that $kG(\gamma_j)=G(\gamma_j/k)$. Hence, 
$\frac{k}{\gamma_{j}}-k \lfloor \frac{1}{\gamma_{j}} \rfloor=\frac{k}{\gamma_{j}}-\lfloor \frac{k}{\gamma_{j}} \rfloor$, 
or $k \lfloor \frac{1}{\gamma_{j}} \rfloor=\lfloor \frac{k}{\gamma_{j}} \rfloor=\lfloor \frac{1}{\gamma'_{j}} \rfloor$.
Thus, $ka_{j+1}=a_{\sigma(j+1)}$. A similar argument establishes that $k=\frac{a_j}{a_{\sigma(j)}}$ for all even $j$. 
\end{proof}

\begin{corollary}
Any perfect $(\sigma,k)$-permutiple, $[a_0; a_1, \ldots, a_n]$, must have an even number of digits.
Moreover, every perfect permutiple is continuant-preserving. 
\label{perf_even}
\end{corollary}

\begin{proof}
 Since $|\sigma|$ is even and divides $n+1$ by Theorem \ref{derangement}, the first statement holds. 
 Thus, $n-1$ is even, so that by Theorem \ref{perf_equiv}, 
 $\gamma_0 \gamma_1 \cdots \gamma_{n-1}=k \gamma'_0 \gamma'_1 \cdots \gamma'_{n-1}$.
 Theorem \ref{contin_pres_equiv} establishes the result. 
 \end{proof}
 
\subsection{Perfect reverse multiples}
We may now find all perfect $k$-reverse multiples.
By Corollary \ref{perf_even}, $n+1$ is even. Thus, Equation \ref{parameters} gives us the relation
$s_j=k^{-\frac{1}{2}[(-1)^j+(-1)^{\sigma(j)}]}s_{\sigma(j)}
=k^{-\frac{1}{2}[(-1)^j+(-1)^{n-j}]}s_{n-j}=s_{n-j}.$
Thus, any perfect $k$-reverse multiple must have the form
$$[ks_0, s_1, ks_2, s_3, \ldots, k^{\frac{1}{2}[1+(-1)^{(n+1)/2}]} s_{(n+1)/2},
k^{-\frac{1}{2}[-1+(-1)^{(n+1)/2}]}s_{(n+1)/2},\ldots,ks_3, s_2, ks_1,s_0].$$

\subsection{Perfect cyclic permutiples}
Letting $\psi$ be the $(n+1)$-cycle $(0,1,2,\cdots,n)$, determining cyclic permutiples which are perfect amounts to solving
$s_j=k^{-\frac{1}{2}[(-1)^j+(-1)^{\psi^{\ell}(j)}]}s_{\psi^{\ell}(j)}$. 
When $\ell$ is odd, this is easy since $s_j=s_{\psi^{\ell}(j)}$.
Suppose then that $\ell$ is even. Then, since $n+1$ is even,
$\psi^{\ell r}(j)$ is either strictly even or strictly odd for all  $0 \leq j \leq n$ and $r \geq 0$.
Therefore, Equation \ref{orbits} cannot be satisfied,
and so there are no perfect cyclic permutiples for which $\ell$ is even.
Thus, we have found all cyclic permutiples which are perfect. 

\begin{example}
 Every perfect, cyclic, 6-digit $(\psi^3,k)$-permutiple has the form $[ks_0; s_1, ks_2,s_0,ks_1,s_2]$.
\end{example}

\section{Symmetric permutiples}
Aside from being continuant-preserving, all of the $(\sigma,k)$-permutiples we have considered so far 
have another property in common: symmetric products of digits are preserved by $\sigma$. 
For example, the perfect permutiple
$$[a_0;a_1,a_2,a_3,a_4,a_5]
=[3; 1, 9, 1, 3, 3]
=3 \cdot [1; 3, 3, 3, 1, 9]
=3 \cdot [a_{\sigma(0)};a_{\sigma(1)},a_{\sigma(2)},a_{\sigma(3)},a_{\sigma(4)}]$$
satisfies
$a_0a_5=9=a_{\sigma(0)}a_{\sigma(5)}$, $a_1a_4=3=a_{\sigma(1)}a_{\sigma(4)}$, and $a_2a_3=9=a_{\sigma(2)}a_{\sigma(3)}$.
Examples such as these are the motivation for the next definition.

\begin{definition}
A $(\sigma,k)$-permutiple, $[a_0; a_{1}, \ldots, a_n]$, is {\it symmetric} provided
$a_j a_{n-j}=a_{\sigma(j)} a_{\sigma(n-j)}$ for all $0 \leq j \leq n$.
\label{symmetric_definition}
\end{definition}

Clearly, all reverse multiples are symmetric.
Also, since any perfect permutiple must have an even number of digits by Corollary \ref{perf_even},
we may write
$$k=\frac{a_0}{a_{\sigma(0)}}=\frac{a_{\sigma(1)}}{a_1}
=\frac{a_2}{a_{\sigma(2)}}=\frac{a_{\sigma(3)}}{a_3}
=\cdots
=\frac{a_{n-3}}{a_{\sigma(n-3)}}=\frac{a_{\sigma(n-2)}}{a_{n-2}}
=\frac{a_{n-1}}{a_{\sigma(n-1)}}=\frac{a_{\sigma(n)}}{a_n}.
$$
Hence, all perfect permutiples are also symmetric. We state these observations formally.

\begin{theorem}
 All reverse multiples and perfect permutiples are symmetric.
 \label{perf_rev_implies_sym}
\end{theorem}

However, generally speaking, symmetric permutiples are richer than only reverse multiples and perfect permutiples;
the permutiple $[4; 2, 1, 8, 1, 2] = 3 \cdot [1; 2, 4, 2, 1, 8]$
provides us with an example of a symmetric permutiple which is neither a reverse multiple, nor perfect.

We also note that there are an abundance of examples which are not symmetric 
(and therefore neither a reverse multiple nor perfect). 
The permutiple $[9; 3, 2, 8, 2]= 4 \cdot [2; 3, 9, 2, 8]$ provides us with an example which is not symmetric.

\section{Permutiples of four digits}\label{4digs}

Using the machinery we have developed thus far, we now consider the 4-digit case.

\begin{theorem}
Every symmetric 4-digit permutiple is continuant-preserving.
\label{sym_contin_pres}
\end{theorem}

\begin{proof}
 Let $r=[a_0;a_1,a_2,a_3]$ be a symmetric $(\sigma,k)$-permutiple. 
 In general, $\frac{K_4(x_0,x_1,x_2,x_3)}{x_0 x_1 x_2 x_3}
 =1 + \frac{1}{x_0 x_1} + \frac{1}{x_1 x_2} + \frac{1}{x_2 x_3} + \frac{1}{x_0 x_1 x_2 x_3}.$
 Hence, 
 $$\frac{p_3}{p'_3}=\frac{K_4(a_0,a_1,a_2,a_3)}{K_4(a_{\sigma(0)},a_{\sigma(1)},a_{\sigma(2)},a_{\sigma(3)})}
 =\frac{1 + \frac{1}{a_0 a_1} + \frac{1}{a_1 a_2} + \frac{1}{a_2 a_3} + \frac{1}{a_0 a_1 a_2 a_3}}
       {1 + \frac{1}{a_{\sigma(0)} a_{\sigma(1)}} + \frac{1}{a_{\sigma(1)} a_{\sigma(2)}} + \frac{1}{a_{\sigma(2)} a_{\sigma(3)}} + \frac{1}{a_0 a_1 a_2 a_3}}.$$
But, since $r$ is symmetric, $a_{1}a_{2}=a_{\sigma(1)}a_{\sigma(2)}$. Thus, we may rewrite the equation above as 
 $$
 \frac{p_3}{p'_3}
 =\frac{1 + \frac{1}{a_0 a_1} + \frac{1}{a_1 a_2} + \frac{1}{a_2 a_3} + \frac{1}{a_0 a_1 a_2 a_3}}
       {1 + \frac{1}{a_1 a_2} + \frac{1}{a_0 a_1 a_2 a_3} + \frac{1}{a_{\sigma(0)} a_{\sigma(1)}} + \frac{1}{a_{\sigma(2)} a_{\sigma(3)}}}.$$       
It follows that 
$$ \frac{p_3}{p'_3}
<\frac{1 + \frac{1}{a_0 a_1} + \frac{1}{a_1 a_2} + \frac{1}{a_2 a_3} + \frac{1}{a_0 a_1 a_2 a_3}}
       {1 + \frac{1}{a_1 a_2} + \frac{1}{a_0 a_1 a_2 a_3} }       
<1+\frac{\frac{1}{a_0 a_1} + \frac{1}{a_2 a_3}}{1 + \frac{1}{a_1 a_2} + \frac{1}{a_0 a_1 a_2 a_3}}.
$$ 
Now, $a_0 \geq 2$ since $a_0>a_{\sigma(0)}$ for any permutiple, and $a_3 \geq 2$ since our representation is canonical. 
Therefore, 
$ \frac{p_3}{p'_3}
<1+\frac{1}{1 + \frac{1}{a_1 a_2}+\frac{1}{a_0 a_1 a_2 a_3}}
<2$,
and the result follows by an application of Theorem \ref{basic}.
\end{proof}

\begin{theorem}
 A 4-digit $(\sigma,k)$-permutiple is symmetric if and only if it is either perfect or a reverse multiple.
 \label{4-digit_equiv_sym}
\end{theorem}

\begin{proof}
 The reverse implication holds for any number of digits by Theorem \ref{perf_rev_implies_sym}. 
 Suppose then that $r=[a_0;a_1,a_2,a_3]$ is a
 symmetric $(\sigma,k)$-permutiple so that $a_0 a_{3} =a_{\sigma(0)} a_{\sigma(3)}$.
 Thus, since $r$ is continuant-preserving by Theorem \ref{sym_contin_pres}, it follows by definition that
 \begin{equation}
   a_0 a_1+a_2 a_3=a_{\sigma(0)}a_{\sigma(1)}+a_{\sigma(2)}a_{\sigma(3)}.
   \label{4-digit_contin_pres_sym_equation_1}
 \end{equation}
 
 Now, if $r$ is a reverse multiple, then we are finished.
 Otherwise, suppose that $r$ is not a reverse multiple. 
 
 If $a_{\sigma(3)}=a_0$, then by symmetry $a_{\sigma(0)}=a_3$.
 It would then follow that $a_{\sigma(1)} a_{\sigma(2)}=a_1 a_{2}$, but since $r$ is not a reverse multiple,
 the only possibility is that $a_{\sigma(1)}=a_1$ and $a_{\sigma(2)}=a_2$. 
 Equation \ref{4-digit_contin_pres_sym_equation_1} would then imply that $a_0=a_3$ and $a_1=a_2$.
 In other words, the continued fraction expansion of $r$ is palindromic.
 The above facts imply that $r=r'$, but the only way this could be is if $k=1$. 
 Therefore,  $a_{\sigma(3)} \neq a_0$.
 Moreover, if $a_{\sigma(3)}=a_3$, then $a_{\sigma(0)} = a_0$ by symmetry.
 But this is also impossible since $a_0>a_{\sigma(0)}$ for any permutiple. 
 Thus, either $a_{\sigma(3)}=a_1$ or $a_{\sigma(3)}=a_2$. That is, we have shown that $\sigma(3) \neq 0$ and $\sigma(3) \neq 3$.
 
 If we assume that $a_{\sigma(0)}=a_3$, then, by symmetry, it follows that $a_{\sigma(3)}=a_0$ which contradicts
 the conclusion above. Therefore, $a_{\sigma(0)} \neq a_3$. Thus, again since $a_0>a_{\sigma(0)}$,
 we have that $a_{\sigma(0)}=a_1$ or $a_{\sigma(0)}=a_2$. In other words, $\sigma(0) \neq 3$ and $\sigma(0) \neq 0$.
 
 We have shown that in any case, $a_0 a_3=a_{\sigma(0)}a_{\sigma(3)}=a_1a_2$, which gives rise to four possible cases.

 Case 1: If $a_{\sigma(0)}=a_1$, $a_{\sigma(1)}=a_0$, $a_{\sigma(2)}=a_3$, and $a_{\sigma(3)}=a_2$,
 then $\frac{a_0}{a_{\sigma(0)}}=\frac{a_{\sigma(1)}}{a_1}$ and $\frac{a_2}{a_{\sigma(2)}}=\frac{a_{\sigma(3)}}{a_3}$. 
 By symmetry, we then have
 $\frac{a_0}{a_{\sigma(0)}}=\frac{a_{\sigma(1)}}{a_1}=\frac{a_2}{a_{\sigma(2)}}=\frac{a_{\sigma(3)}}{a_3}$.
 Now, since $r$ is continuant-preserving, it follows by Theorem \ref{contin_pres_equiv} that
  \begin{equation}
   a_{\sigma(1)}a_{\sigma(2)}a_{\sigma(3)}+a_{\sigma(1)}+a_{\sigma(3)}=k(a_1 a_2 a_3 + a_1 +a _3).
   \label{4-digit_contin_pres_sym_equation_2}
 \end{equation} 
 Thus, $a_{0}a_{3}a_{2}+a_{0}+a_{2}=k(a_1 a_2a_3+a_1+a_3)$.
 Therefore, since $a_0a_3=a_1a_2$ as also shown above, we have $a_{1}a_{2}^2+a_{0}+a_{2}=k(a_1 a_2a_3+a_1+a_3)$,
 which becomes $a_0-ka_1+(1+a_1a_2)(a_2-ka_3)=0$. 
 For a contradiction, suppose $a_0 > k a_1$. 
 Then it would also have to be that $a_2 > k a_{3}$ since $a_0a_3=a_1a_2$. But this implies that  
 $a_0-ka_1+(1+a_1a_2)(a_2-ka_3)>0$ which contradicts the previous equation. Assuming $a_0 < k a_1$ similarly
 leads to a contradiction. Therefore, in order for the above equation to hold, it can only be that $a_0 = k a_1$
 and $a_2 = k a_3$.
 Hence, $\frac{a_0}{a_{\sigma(0)}}=\frac{a_{\sigma(1)}}{a_1}=\frac{a_2}{a_{\sigma(2)}}=\frac{a_{\sigma(3)}}{a_3}=k$, 
 and the result holds for Case 1.

 Case 2: If $a_{\sigma(0)}=a_1$, $a_{\sigma(1)}=a_3$, $a_{\sigma(2)}=a_0$, and $a_{\sigma(3)}=a_2$,
 then, by Equation \ref{4-digit_contin_pres_sym_equation_1}, $a_0 a_1+a_2 a_3=a_{1}a_{3}+a_{0}a_{2}$.
 It follows that $(a_1-a_2)(a_3-a_{0})=0$ so that
 $a_0=a_3$ and $a_1=a_2$.
 Thus, since $a_0 a_3=a_1a_2$ as demonstrated above, we conclude that all the digits are equal. 
 However, this implies that $k=1$.
 Therefore, Case 2 is impossible.

 Case 3: If $a_{\sigma(0)}=a_2$, $a_{\sigma(1)}=a_0$, $a_{\sigma(2)}=a_3$, and $a_{\sigma(3)}=a_1$,
 then, again by Equation \ref{4-digit_contin_pres_sym_equation_1}, $a_0 a_1+a_2 a_3=a_{2}a_{0}+a_{3}a_{1}$
 which leads to the same conclusion as Case 2, so that Case 3 is also impossible.

 Case 4: Finally, suppose that $a_{\sigma(0)}=a_2$, $a_{\sigma(1)}=a_3$, $a_{\sigma(2)}=a_0$, and $a_{\sigma(3)}=a_1$.
 Then, by Equation \ref{4-digit_contin_pres_sym_equation_2}, $a_{3}a_{0}a_{1}+a_{3}+a_{1}=k(a_1 a_2a_3+a_1+a_3)$.
 It follows that
 $a_0a_3=a_1a_2=\frac{(k-1)(a_1+a_3)}{a_1-ka_3}.$ 
 Therefore, $a_1$ must be strictly greater than $ka_3$.
 Using Equation \ref{4-digit_contin_pres_sym_equation_2} again, and reducing modulo $a_1$, we see that
 $(k-1)a_3 \equiv 0$ (mod $a_1$). Then, for some $\alpha$, $(k-1)a_3=\alpha a_1$. 
 But then $a_1>k a_3$ implies that $a_1-a_3>(k-1)a_3=\alpha a_1$. Thus, $1-\frac{a_3}{a_1}>\alpha$,
which means that $\alpha \leq 0$. By the above, it follows that $a_3 \leq 0$, which is a contradiction.   
 Case 4 is therefore impossible. 
 
 Since Case 1 is the only possibility, the result is established.
\end{proof}

Computer generated evidence strongly suggests that all 4-digit permutiples are symmetric.
We suspect that a tedious argument like the one above, and perhaps an application of Theorem \ref{sym_contin_pres},
may be involved.
However, despite our best efforts, we have not been able to establish this claim. 
We therefore leave the following conjecture.

\begin{conjecture}
 Every 4-digit $(\sigma,k)$-permutiple is symmetric.
\end{conjecture}

If the above is indeed true, we would have by Theorem \ref{4-digit_equiv_sym} the following corollary.

\begin{conjecture}
 Any 4-digit $(\sigma,k)$-permutiple is either perfect, a $k$-reverse multiple, or both.
\end{conjecture}

\section{Concatenation}

With several permutiple types and many permutiple examples in hand, we now consider ways of constucting new permutiples from old. 
Let $c_1$ and $c_2$ be any finite simple continued fractions,  
$[b_0;b_1,b_2,\ldots,b_n]$ and $[b'_0;b'_1,b'_2,\ldots,b'_m]$, respectively.
We let $c_1 \circ c_2$ denote the continued fraction obtained by concatenating the digit strings of $c_1$ and $c_2$,
that is, $[b_0;b_1,b_2,\ldots,b_n, b'_0,b'_1,b'_2,\ldots,b'_m]$, and call this quantity
the {\it concatenation of $c_1$ and $c_2$}.

For convenience, we also introduce a new notation for the continuant of a $(\sigma,k)$-permutiple, $r$,
with continued fraction expansion $[a_0;a_1,\ldots,a_n]$;
we define $\langle r \rangle$ to mean $K_{n+1}(a_0,a_1,\ldots,a_n)$. Furthermore,
we define $\langle {}_{-}r \rangle$ to mean $K_n(a_1,\ldots,a_n)$, and 
$\langle r_{-} \rangle$ to mean $K_n(a_0,a_1,\ldots,a_{n-1})$.

\begin{theorem}
 Let $r=[a_0; a_{1}, \ldots, a_n]$ be a continuant-preserving $(\sigma,k)$-permutiple with 
 $r'=[a_{\sigma(0)}; a_{\sigma(1)}, \ldots, a_{\sigma(n)}]$, such that $q_{n-1}=q'_{n-1}$ and $p_{n-1}=kp'_{n-1}$.
 Also, let $s=[b_0;b_1,\ldots,b_m]$ be any continuant-preserving $(\tau,k)$-permutiple with 
 $s'=[b_{\tau(0)};b_{\tau(1)},\ldots,b_{\tau(m)}]$.
 Then the number $r \circ s$  is also a continuant-preserving permutiple with $r \circ s = k(r'\circ s')$.
 \label{concatenation}
\end{theorem}

\begin{proof}
Since $r$ and $s$ are continuant-preserving, we have $\langle r \rangle=\langle r' \rangle$,
$\langle s \rangle=\langle s' \rangle$, $\langle {}_{-}r' \rangle= k \langle {}_{-}r \rangle$,
and $\langle {}_{-}s' \rangle=k\langle {}_{-}s \rangle$. 
Also, writing $q_{n-1}=q'_{n-1}$ and $p_{n-1}=kp'_{n-1}$ in terms of continuants, we have
$\langle{}_{-}r_{-}\rangle=\langle{}_{-}r'_{-}\rangle$ and $\langle r_{-} \rangle=k \langle r'_{-} \rangle$, respectively.
Using the above and properties of continuants of concatenations \cite{benjamin,hensley}, we have
$
\langle {}_{-}r'\circ s' \rangle
=\langle {}_{-}r' \rangle \langle s' \rangle + \langle {}_{-}r'_{-} \rangle \langle {}_{-}s' \rangle
=\langle {}_{-}r' \rangle \langle s' \rangle + \langle {}_{-}r'_{-} \rangle \langle {}_{-}s' \rangle
=k \langle {}_{-}r \rangle \langle s \rangle + \langle {}_{-}r_{-} \rangle k\langle {}_{-}s \rangle
=k\langle {}_{-}r\circ s \rangle.
$
Therefore, 
$\frac{\langle r \circ s \rangle}{\langle {}_{-}r\circ s \rangle}
=k\frac{\langle r' \circ s' \rangle}{\langle {}_{-}r' \circ s' \rangle}$ so that  
$r \circ s$ is a permutiple.
Now, again using properties of continuants, 
$\langle r' \circ s' \rangle 
=\langle r' \rangle \langle s' \rangle + \langle r'_{-} \rangle \langle {}_{-}s' \rangle
=\langle r \rangle \langle s \rangle + k \langle r'_{-} \rangle \langle {}_{-}s \rangle
=\langle r \rangle \langle s \rangle + \langle r_{-} \rangle \langle {}_{-}s \rangle
=\langle r' \circ s' \rangle$ so that $r \circ s$ is continuant-preserving.
\end{proof}
  
As the reader shall presently see, permutiples which satisfy the hypothesis of Theorem \ref{concatenation} 
have other special properties, and for the sake of tidier exposition, we shall give them a name. 
  
\begin{definition}
 We shall call any continuant-preserving permutiple with the property that $q_{n-1}=q'_{n-1}$ and $p_{n-1}=kp'_{n-1}$ 
 a {\it Landess} \footnote{In honor of M. J. Landess, a great mathematics educator.} permutiple.
\end{definition}

\begin{remark}
 Since continuants are invariant under reversal, it is not difficult to show that all $k$-reverse multiples satisfy
 $\langle{}_{-}r_{-}\rangle=\langle{}_{-}r'_{-}\rangle$ and $\langle r_{-} \rangle=k \langle r'_{-} \rangle$.
 Thus, all $k$-reverse multiples are examples of Landess permutiples.
\end{remark}

\begin{theorem}
The concatenation of any two Landess permutiples (both with multiplier $k$) is again
a Landess permutiple (with multiplier $k$).
\label{landess_concat}
\end{theorem}

\begin{proof}
 Let $r$ and $s$ be two Landess permutiples, both with multiplier $k$.
 By Theorem \ref{concatenation}, $r \circ s$ is a continuant-preserving permutiple with multiplier $k$,
 and since $s$ is a Landess permutiple, the result follows.  
\end{proof}

\begin{remark}
Defining $L_k$ as the collection of all Landess permutiples with multiplier $k$, and
appending to this collection an ``empty'' permutiple, we may then construct the free monoid $L_k^{\ast}$
of $k$-Landess permutiples.
\end{remark}

\begin{corollary}
 The concatenation of any $k$-reverse multiple with itself is again a $k$-reverse multiple.
\end{corollary}

\begin{proof}
 Since every $k$-reverse multiple is a Landess permutiple, we may apply Theorem \ref{landess_concat}
 to deduce that the concatenation is also a permutiple with multiplier $k$. The rest is a straightforward 
 argument.
\end{proof}

From the above corollary, it follows by induction that a palindromic concatenation of $k$-reverse multiples
is again a $k$-reverse multiple. More formally we have the following.

\begin{corollary}
 Suppose $r_0$, $r_1$, $\ldots$, $r_m$ are all $k$-reverse multiples such that $r_j=r_{m-j}$ for all $0 \leq j \leq m$,
 then $r_0 \circ r_1 \circ \cdots \circ r_m$ is also a $k$-reverse multiple. 
\end{corollary}

 The next theorem reveals that reverse multiples are not the only examples of Landess permutiples.

\begin{theorem}
 Every perfect permutiple is also a Landess permutiple.
\end{theorem}

\begin{proof}
 Let $r=[a_0;a_1,\ldots,a_n]$ be a perfect $(\sigma,k)$-permutiple.
 By Corollary \ref{perf_even}, every perfect permutiple is continuant-preserving. 
 Thus, $p_n=p'_n$ so that $a_n p_{n-1}+p_{n-2}=a_{\sigma(n)} p'_{n-1}+p'_{n-2}$.
 Since $r$ must have an even number of digits, $n$ is odd, and since $r$ is perfect, we have that 
 $a_{\sigma(n)}=ka_n$. Therefore, $a_n p_{n-1}+p_{n-2}=a_n k p'_{n-1}+p'_{n-2}$.
 Now, for any perfect permutiple, it is a straightforward induction argument to prove directly using the recursive definition 
 of the continuant that 
 $p_{j-1}=K_j(a_0, a_1,\ldots, a_{j-1})=kK_j(a_{\sigma(0)}, a_{\sigma(1)}, \ldots a_{\sigma(j-1)})=kp'_{j-1}$
 when $j$ is odd, and 
 $p_{j}=K_{j+1}(a_0, a_1,\ldots, a_{j})=K_{j+1}(a_{\sigma(0)}, a_{\sigma(1)}, \ldots a_{\sigma(j)})=p'_j$
 when $j$ is even. Therefore, $p_{n-2}=p'_{n-2}$.
 Hence, by the above, we have shown that $p_{n-1}=k p'_{n-1}$.
 
 A similar argument proves that $q_{n-1}=q'_{n-1}$.
\end{proof}

\begin{corollary}
 The concatenation of any two perfect permutiples is again a perfect permutiple.
\end{corollary}

Landess permutiples are not always perfect or reverse multiples
as shown by $[2; 1, 5, 1, 2] = 2 \cdot [1; 2, 2, 1, 5]$.
The reader will also notice that this is a non-symmetric example. 
Also, not every permutiple is a Landess permutiple as the example
$[11; 1, 10, 2, 3]=9 \cdot [1; 3, 11, 10, 2]$ proves.

\section{Future work and concluding remarks}
By now, the reader is sure to have noticed the lack of examples of permutiples which are not continuant-preserving.
This is because, despite hours of computer search time, we have been unsuccessful in finding one. 
Yet, a demonstration that none exist has so far eluded us. 

The lack of success in finding a counterexample does not necessarily cast overwhelming doubt upon the possibility 
that one exists. Given the diversity of examples and permutiple types we have already encountered, 
all of which occur for a relatively low number of digits, a
counterexample would not necessarily be surprising if we were to expand our search.
However, with the present results we have, we are still limited to either brute force, or random search.
Consequently, either increasing the bound on the digits, or the number of digits, 
incurs a very substantial increase in search time. 
We also  mention that when comparing expressions of the form 
$R=[b_0;b_1,\ldots,b_n]$ and $R'=[b_{\sigma(0)}; b_{\sigma(1)}, \ldots, b_{\sigma(n)}]$
(not necessarily permutiples), we can construct continued fractions which make the ratios of numerators of $n$th 
convergents as large as we please. Large ratios typically occur when there are large differences between digits.  
These considerations, combined with computational limitations already mentioned, mean that a counterexample
may very well exist beyond the bounds we have been able to reasonably check. 
However, as of the writing of this paper, we have still been unable to find one. 
We therefore formally pose the following question and leave it as an open problem.

\begin{center}
{\it Are all permutiples continuant-preserving?}
\end{center}

We also mention that although not every Landess permutiple is symmetric,
computer-generated evidence does suggest that every symmetric permutiple is a Landess permutiple.
However, to show this is the case, we must have that all symmetric permutiples are continuant-preserving.
We also leave these as open questions.

\begin{center}
{\it Are all symmetric permutiples also Landess permutiples?}
\end{center}

\begin{center}
{\it Are all symmetric permutiples even continuant-preserving?}
\end{center}

For the finite case, we summarize the permutiple types we have observed so far, as well as the above conjectures, 
with the figure below. A dashed border indicates that it is unknown if the pictured containment holds. 
In the figure we have also included examples from the above exposition which prove strict 
containment given that the containment actually holds.
\begin{center}
\includegraphics[width=10cm]{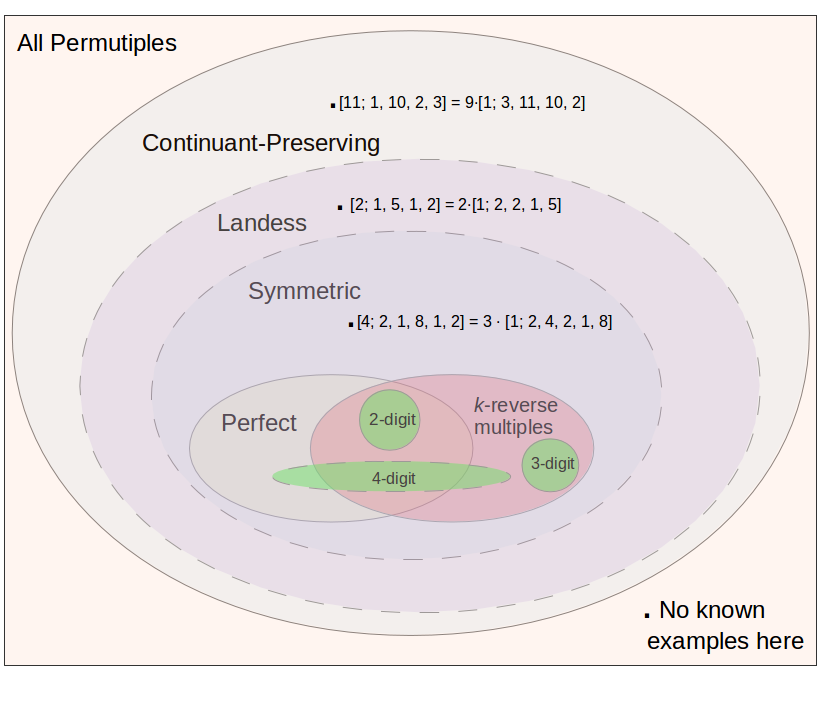} 
\end{center}

In addition to the questions we have already left to the reader, there are clearly other avenues of inquiry to pursue. 
In particular, we mention the permutiple problem in the settings of both general and infinite continued fraction representations.

Examples of the latter readily come to mind; the real number $1+\sqrt{3}$ is an infinite simple
continued fraction permutiple since $[2;1,2,1,\ldots]=1+\sqrt{3}=2 \cdot \left(\frac{1+\sqrt{3}}{2}\right)=2 \cdot [1;2,1,2,\ldots]$.
In fact, with a little more effort, and some classical results \cite{olds}, we can show that any reduced quadratic surd 
$r=\frac{a+\sqrt{b}}{c}$
is an infinite permutiple precisely when $\frac{b-a^2}{c}$ is an integer since the multipler $k$ is equal to this quantity.

The notion of perfect permutiples carries over very naturally to the infinite setting; 
Definition \ref{perfect_def} requires essentially no modification.
Also, we only need to slightly change the statement of Theorem \ref{perf_equiv} to accomodate the above definitions,
and its proof requires no modification. The periodic examples mentioned above are perfect.
  
Infinite concatenations of Landess permutiples can yield non-periodic infinite examples.
For example, the continued fraction $[k s_0;s_0,ks_1,s_1,ks_2,s_2,\ldots]$ is an infinite perfect $(\sigma,k)$-permutiple
where $\sigma(j)=j+(-1)^{j}$, and each $s_j>0$ is a free integer parameter for all $j \geq 0$.
This example is a concatenation of 2-digit perfect permutiples.
Thus, for example, $[2;1,8,4,32,16,\ldots]=2 \cdot [1;2,4,8,16,32,\ldots]$ is an infinite perfect permutiple
which is non-periodic.  

What it means for an infinite simple continued fraction to be continuant-preserving
is not difficult to generalize either, although it does require a little more work. 
Numerators of $n$th convergents, for example 
$K_{n+1}(2,1,2,1,\ldots,a_n)$ and $K_{n+1}(1,2,1,2,\ldots,a_{\sigma(n)})$,
are not equal in all cases. What matters is that they are equal for an infinite number of cases;
we say that an infinite $(\sigma,k)$-permutiple, $[a_0; a_{1}, a_2 \ldots]$, is {\it asymptotically continuant-preserving} if 
$$\liminf_{n\rightarrow \infty} \left| K_{n+1}(a_0,a_1,\ldots,a_n) -K_{n+1}(a_{\sigma(0)},a_{\sigma(1)},\ldots,a_{\sigma(n)})\right|=0.$$
Analogous to the finite case, the reader will notice that in the case of purely periodic permutiples
considered above, the numerators of the corresponding reduced quadratic surds are equal.
It is also not difficult to prove an infinite analogue to Theorem \ref{contin_pres_equiv}.
 \begin{theorem}
  For any infinite $(\sigma,k)$-permutiple, $r=[a_0;a_1,a_2,a_3 \ldots]$, the following are equivalent:
  \begin{enumerate}
    \item $r$ is asymptotically continuant-preserving, 
    \item $\liminf\limits_{n\rightarrow \infty} |K_{n}(a_{\sigma(1)},\ldots,a_{\sigma(n)})-k K_{n}(a_1,\ldots,a_n)|=0$,
    \item $\liminf\limits_{n\rightarrow \infty} |\gamma_0 \gamma_1 \cdots \gamma_n-k \gamma'_0 \gamma'_1 \cdots \gamma'_n|=0$.
  \end{enumerate}
 \end{theorem}
 For any perfect permutiple, item {\it 2} of Theorem \ref{perf_equiv} gives us that 
 $\gamma_0 \gamma_1 \cdots \gamma_n=k \gamma'_0 \gamma'_1 \cdots \gamma'_n$
  for all even $n$. An application of the above theorem gives us the following corollary. 
  \begin{corollary}
     Every infinite perfect permutiple is asymptotically continuant-preserving.
  \end{corollary}   

Generalizations of other notions presented in the finite case are possible, but we shall leave them 
for the reader to rediscover and develop more fully on their own. 

\section{Acknowledgment}
The author would like to thank Casey Bonavia of Columbia College for proofreading the manuscript prior to 
its submission for publication.

\bigskip
\hrule
\bigskip

\noindent 2010 {\it Mathematics Subject Classification}: Primary 11A55; Secondary 37B10, 11B85.

\noindent \emph{Keywords:} continued fraction, permutiple, reverse multiple, palintiple.

\bigskip
\hrule
\bigskip

\noindent (Concerned with sequences
\seqnum{A001232}, \seqnum{A002530}, \seqnum{A008918}, \seqnum{A008919}, \seqnum{A031877}, \seqnum{A146326}, \seqnum{A169824}.)


\begin{thebibliography}{3}

\bibitem{benjamin}
A. T. Benjamin and D. Zeilberger, Pythagorean primes and palindromic continued fractions, {\it Integers} {\bf 5}
 (2005), \#A30.

\bibitem{guttman}
S. Guttman, On cyclic numbers, \textit{The Amer. Math. Monthly} \textbf{41} (1934), 159--166.

\bibitem{hensley}
D. Hensley, {\it Continued Fractions}, World Scientific, 2006.

\bibitem{holt}
B. V. Holt, Some general results and open questions on palintiple numbers, {\it Integers} {\bf 14} (2014), \#A42. 

\bibitem{holt_2}
B. V. Holt, Derived palintiple families and their palinomials, {\it Integers} {\bf 16} (2016), \#A27. 

\bibitem{holt_3}
B. V. Holt, On permutiples having a fixed set of digits.
Preprint, 2015, available at \url{https://arxiv.org/pdf/1511.02033.pdf}.
Submitted to {\it Integers} for review.

\bibitem{kalman}
D. Kalman, Fractions with cycling digit patterns,
\textit{The College Math. J.} \textbf{27} (1996), 109--115.

\bibitem{kendrick_1}
L. H. Kendrick, Young graphs: 1089 et al., {\it J. Integer Seq.} {\bf 18} (2015), Article 15.9.7.

\bibitem{kendrick_2}
L. H. Kendrick, Data for Young graphs and their isomorphism classes, 
\url{https://sites.google.com/site/younggraphs/home}.

\bibitem{olds}
C. D. Olds, \textit{Continued Fractions}, Random House, 1963.

\bibitem{pudwell}
L. Pudwell, Digit reversal without apology, {\it Math. Mag.} {\bf 80} (2007), 129--132.

\bibitem{sloane}
N. J. A. Sloane, 2178 and all that, {\it Fibonacci Quart.} {\bf 52} (2014), 99--120.


\bibitem{young_1}
A. L. Young, $k$-reverse multiples, {\it Fibonacci Quart.} {\bf 30} (1992), 126--132.

\end{thebibliography}
\end{document}